\documentclass[11pt]{article}
\usepackage{indentfirst}
\usepackage{amssymb,latexsym,amsmath,amsbsy,amsthm,amsxtra,amsgen,graphicx,makeidx,dsfont,longtable,bbm}
\usepackage{hyperref}
\usepackage{tikz}

\newtheorem{remark}{Remark}[section]

\newtheorem{lemma}{Lemma}[section]
\newtheorem{proposition}{Proposition}[section]
\newtheorem{theorem}{Theorem}[section]

\def\Var{\textup{Var}}
\def\Cov{\textup{Cov}}

\def\R{\mathbb{R}}
\def\C{\mathbb{C}}
\def\E{\mathbb{E}}
\def\P{\mathbb{P}}

\begin{document}
	\title{The site frequency spectrum for coalescing Brownian motion}
	\author{Yubo Shuai \\
		University of California at San Diego}
	\maketitle
	
	
	
	\vspace{-.6in}
	\begin{abstract}
		We consider an expanding population on the plane. The genealogy of a sample from the population is modelled by coalescing Brownian motion on the circle. We establish a weak law of large numbers for the site frequency spectrum in this model. A parallel  result holds for a localized version where the genealogy is modelled by coalescing Brownian motion on the line.
	\end{abstract}
	
	\section{Introduction}
	
	In population genetics, one is often interested in the mutations along the DNA sequences in a sample from a population. The site frequency spectrum is commonly used to summarize the mutational data. In a sample of size $n$, the site frequency spectrum consists of $M_m$ for $m=1,2,\dots, n-1$ where $M_m$ is the number of mutations inherited by exactly $m$ individuals in the sample. There is an extensive literature on the exact and asymptotic behavior of the site frequency spectrum for various population models. For models with fixed population size, Fu and Li \cite{fu1993statistical} computed the expected site frequency spectrum for a population whose genealogy is given by Kingman's coalescent. This computation was generalized to $\Lambda$-coalescents by Birkner, Blath and Eldon \cite{birkner2013statistical}. In the special case of the Bolthausen–Sznitman coalescent, Diehl and Kersting obtained laws of large numbers for the site frequency spectrum \cite{10.1214/19-AAP1462}. The computation for expected site frequency was further generalized to $\Xi$-coalescents by Spence, Kamm and Song \cite{spence2016site} and  Blath et al. \cite{blath2016site}. For models with exponentially growing population size, the asymptotics of the expected site frequency spectrum were
	obtained by Durrett \cite{durrett2013population} and the exact formula when the whole population is sampled was obtained by Gunnarsson, Leder and Foo \cite{gunnarsson2021exact}. Schweinsberg and Shuai established the asymptotic normality for the site frequency spectrum in \cite{schweinsberg2023asymptotics} based on the methods developed in \cite{johnson2023estimating}.
	
	The results mentioned above assume a well-mixed population and no spatial constraint is imposed. De and Durrett \cite{de2007stepping} considered the stepping stone model and observed that the there are more high frequency mutations due to the spatial structure. In this paper, we consider a population whose genealogy is modelled by coalescing Brownian motion with Poissonian mutations along the branches and establish a weak law of large numbers for the site frequency spectrum.  
	
	\subsection{An expanding population model}
	We think of a population on $\R^2$. The ancestor is located at the origin and the $k$th generation live on the circle with radius $k$ centered at the origin. To take the spatial structure into account, the offspring of an individual in the $k$th generation are located in a neighborhood of the parent, in the sense that the angular parts of the offspring and the parent are close. More formally, we fix some non-decreasing function $f:\R\rightarrow\R$ satisfying the periodic condition $f(x+1)=f(x)+1$ and let $\Theta_1,\Theta_2,\dots$ be a sequence of i.i.d random variables, uniformly distributed on $[0,1)$. 
	For an individual $x$ in the $(k+1)$st generation with angular part $\theta_x$, its parent in the $k$th generation has radial part $f_{\Theta_k}(\theta_x)=f(\theta_x-\Theta_k)+\Theta_k$. We are interested in the angular parts of the ancestral lineages of $x$, namely, $\Phi_{j,k}(\theta_x)=f_{\Theta_{j+1}}\circ\dots\circ f_{\Theta_{k}}(\theta_x)$ for $j\le k$. It is worth noting that, if $f$ is strictly increasing, then the angular parts of the ancestral lineages remain distinct. However, under proper scaling, these angular parts converge to the coalescing Brownian motion on the circle. Indeed, Norris and Turner \cite{norris2015weak} embedded $\Phi_{j,k}$ in a continuous time setting. They showed that for any sequence of functions $f_n$ that converges to the identity function appropriately, and for $(\theta_{x_n})_{n=1}^{\infty}\subset [0,1)$, the ancestral lineages of $(\theta_{x_n})_{n=1}^{\infty}$ converge weakly to the coalescing Brownian motion on the circle. Also, a localized version is shown to converge weakly to the coalescing Brownian motion on the line. We will therefore study the site frequency spectrum for population whose genealogy is given by coalescing Brownian motion on the circle or coalescing Brownian motion on the line.

	\subsection{Coalescing Markov process and coalescing Brownian motion}
	Coalescing Brownian motion was first studied by Arratia in his Ph.D. thesis \cite{arratia1979coalescing} at the University of Wisconsin, Madison. More generally, coalescing Markov processes were introduced by Donnelly et al. in the study of the stepping stone model in \cite{donnelly2000continuum}. Heuristically, we have $n$ distinct particles located at $(e_1, \dots , e_n)$ in some state space $E$. These particles evolve as independent Markov processes before two or more particles collide. When a collision occurs, those particles coalesce into one particle and then undergo the same dynamics. We will formalize this for coalescing Brownian motion on the real line or the circle in the next paragraph. For the general setting, we refer the reader to \cite{donnelly2000continuum}. 
	
	Throughout the rest of the paper, for each $x\in\R$, we write $W^x$ for the 1-dimensional Brownian motion starting from $x$. We also assume that $\{W^x\}_{x\in\R}$ are independent. Let the state space $E$ be either $\R$ or $\mathbb{S}^1=\{z\in \C: |z|=1\}$, which we call the linear and circular cases respectively. Using notation in Section 2 of \cite{donnelly2000continuum}, we take the initial positions
	$$
	\boldsymbol{e_n}=(e_{n,1},\dots, e_{n,n}):=\left\{\begin{array}{lc}
		(1/n,\dots, n/n),&\text{ if } E=\R,\\
		(\exp(2\pi i/n),\dots, \exp(2n\pi i/n)), &\text{ if } E=\mathbb{S}^1.
	\end{array}\right.
	$$
	If there is no collision, then the paths of these particles are
	$$
	\boldsymbol{Z^{e_n}}=(Z^{e_{n,1}},\dots, Z^{e_{n,n}}):=\left\{\begin{array}{lc}
		\left(W^{1/n},\dots, W^{n/n}\right),&\text{ if } E=\R,\\
		\left(\exp(2\pi i W^{1/n}),\dots, \exp(2\pi i W^{n/n})\right), &\text{ if } E=\mathbb{S}^1.
	\end{array}\right.
	$$
	To describe the dynamics with collisions, we introduce the coalescence times $\tau_{n,k}$ and the partitions $\Pi_{n,k}$ of $\{1,\dots, n\}$, where $i,j$ are in the same block of $\Pi_{n,k}$ if the $i$th particle coalesces with the $j$th particle no later than $\tau_{n,k}$. For each block of $\Pi_{n,k}$, we use the particle with the smallest index as the representative. Formally, we define $\tau_{n,k}$ and $\Pi_{n,k}$ inductively for $k=0,1,2,\dots, n-1$.  We take $\tau_{n,0}=0$ and $\Pi_{n,0}=\{\{1\},\dots,\{n\}\}$. Given $\tau_{n,k}$ and $\Pi_{n,k}$, we define 
	$$
	\tau_{n,k+1}:=\inf\left\{t>\tau_k: \exists A,A'\in\Pi_{n,k},\ A\neq A',\ Z^{e_{n,\min(A)}}_t=Z^{e_{n,\min(A')}}_t\right\}.
	$$
	Let $A,A'\in\Pi_{n,k}$ be the blocks coalescing at $\tau_{n,k+1}$, i.e. $Z^{e_{n,\min(A)}}_{\tau_{n,k+1}}=Z^{e_{n,\min(A')}}_{\tau_{n,k+1}}$. Then the partition $\Pi_{n,k+1}$ is obtained from $\Pi_{n,k}$ by merging the blocks of $A$ and $A'$:
	$$
	\Pi_{n,k+1}:=\left(\Pi_{n,k}\setminus\{A,A'\}\right) \cup \{A\cup A'\}.
	$$
	The actual position of the $i$th particle at time $t$, denoted by $\check{Z}^{e_{n,i}}_t$, is
	$$
	\check{Z}^{e_{n,i}}_t = Z^{e_{n,j}}_t, \text{ if } \tau_{n,k}\le t< \tau_{n,k+1},\ i\in A\in \Pi_{n,k},\text{ and }\ j=\min(A).
	$$
	For any $A\subseteq\{1,2,\dots,n\}$ with $|A|\ge 2$,
	we define the first coalescence time of $A$ as
	\begin{equation*}\label{Intro: def of coal time in S}
		\tau_{n,A} := \inf\{t\ge 0: \exists i,j\in A,\ i\neq j,\  \check{Z}^{e_{i}}_t=\check{Z}^{e_{j}}_t \}.
	\end{equation*}	
	If $A\subseteq\{1,2,\dots,n\}$ and $|A|=1$, we set $\tau_{n,A}=0$ by convention.
	
	\begin{remark}\label{Remark: equivalence in distribution}
		Note that for any $A\subset\{1,2,\dots,n\}$, the actual positions of the particles in $A$ evolve as independent Brownian motions before the time $\tau_{n,A}$. That is, if we set 
		$$\tau_{A}:= \inf\{t\ge 0: \exists i,j\in A,\ i\neq j,\  {Z}^{e_{i}}_t={Z}^{e_{j}}_t \},$$
		then $\left(\check{Z}^{e_{n,i}}_t, 0\le t\le \tau_{n,A}\right)_{i\in A}$ has the same distribution as $\left(Z^{e_{n,i}}_t, 0\le t\le \tau_{A}\right)_{i\in A}$ , although they may not be equal because of coalescence with with particles not in $A$. 
	\end{remark}
	
	\subsection{The site frequency spectrum of coalescing Brownian motion}
	Given the coalescing Brownian motion $(\check{Z}^{e_{n,i}})_{1\le i\le n}$, one can study the the corresponding genealogical tree on $E\times[0,+\infty)$, where the branches of the tree correspond to the trajectories of $(\check{Z}^{e_{n,i}})_{1\le i\le n}$. See Figure 1 for an example.
	
	\begin{figure}[h]
		\centering
		\begin{tikzpicture}[scale=0.75]
			\draw [very thick] (0,0)--(1.5,2.5);
			\draw [very thick] (1.5,0)--(2.25,1.25);
			\draw [very thick] (3,0)--(2.25,1.25);
			\draw [very thick] (4.5,0)--(5.25,0.75);
			\draw [very thick] (6,0)--(5.25,0.75);
			\draw [very thick] (7.5,0)--(6,2);
			
			\draw [very thick] (2.25,1.25)--(1.5,2.5);
			\draw [very thick] (5.25,0.75)--(6,2);
			
			\draw [very thick, red] (1.5,2.5)--(3.75,4);
			\draw [very thick] (6,2)--(3.75,4);
			
			\node at (0,-0.3){1};
			\node at (1.5,-0.3){2};
			\node at (3,-0.3){3};
			\node at (4.5,-0.3){4};
			\node at (6,-0.3){5};
			\node at (7.5,-0.3){6};
		\end{tikzpicture}
		\caption{The genealogical tree for coalescing Brownian motion on $E=\R$ and $n=6$. Brownian trajectories are represented by straight lines. The red branch supports 3 leaves, namely 1,2, and 3.} \label{fig:treeshape}
	\end{figure}
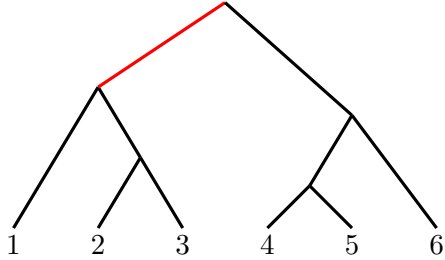
	
	The number of mutations inherited by $m$ individuals $M_m$ is directly related to the length of the branches supporting $m$ leaves in the genealogical tree $L_{n,m}$ (See Figure \ref{fig:treeshape}). Indeed, a mutation along a branch supporting $m$ leaves will be inherited by $m$ individuals in the sample. If we assume that mutations occur with rate $\nu$ along the branches, independently of the Brownian motion, then the conditional distribution of $M_m$ given $L_{n,m}$ is Poisson with mean $\nu L_{n,m}$. For this reason, we will focus on $L_{n,m}$ in this paper.
	
	In the linear case, we say that the $i$th branch supports $m$ leaves at time $t\in[\tau_{n,k},\tau_{n,k+1})$ if $i$ is a representative of a block of size $m$. That is,
	\begin{equation}\label{Intro: linear, def of length of the ith branch that supports m leaves}
		\exists A\in\Pi_{n,k} \text{ such that } |A|=m, \min(A)=i.    
	\end{equation}
	The length of the portion of the $i$th branch that supports $m$ leaves, denoted by $L_{n,i,m}$, is the length of the time interval when \eqref{Intro: linear, def of length of the ith branch that supports m leaves} holds. Note that the only block $A$ for which \eqref{Intro: linear, def of length of the ith branch that supports m leaves} could be true is $\{i,i+1,\dots, i+m-1\}$, provided $i+m-1\le n$. This block is in $\Pi_{n,k}$ if and only if particle $i$ coalesces with particle $i+m-1$, and particles $i$ and $i+m-1$ do not coalesce with any other particles outside of the block. Writing $x^+=\max(x,0)$, we have
	\begin{align}\label{Intro: linear, formula for length of the ith branch that supports m leaves}
		L_{n,i,m}=\left\{\begin{array}{lc}
			\left(\tau_{n,\{m,m+1\}}-\tau_{n,\{1,m\}}\right)^{+},\qquad & i=1,\\
			\left(\tau_{n,{\{i-1,i\}}}\wedge\tau_{n,\{i+m-1,i+m\}}-\tau_{n,\{i,i+m-1\}}\right)^+, \qquad & 2\le i\le n-m,\\
			\left(\tau_{n,\{n-m,n-m+1\}}-\tau_{n,\{n-m+1,n\}}\right)^{+}, \qquad &i = n-m+1,\\
			0,&i\ge n-m+2.
		\end{array}
		\right.
	\end{align}
	
	In the circular case, we want to respect the symmetry of $\mathbb{S}^1$ so that $L_{n,i,m}$ has the same distribution for all $i$. To do this, we identify an integer with its equivalence class in $\{1,\dots,n\}$ modulo $n$ and define $L_{n,i,m}$ to be the the time elapsed for which $\{i,i+1,\dots,i+m-1\}$ is in the partition. Formally, we have
	\begin{equation}\label{Intro: circular, def of length of the ith branch that supports m leaves on the cirlce}
		L_{n,i,m}= \left(\tau_{n,{\{i-1,i\}}}\wedge\tau_{n,\{i+m-1,i+m\}}-\tau_{n\{i,i+m-1\}}\right)^+, \qquad 1\le i\le n.
	\end{equation}
	For both the linear and the circular cases, the total length of the branches that support $m$, leaves, denoted by $L_{n,m}$ is 
	\begin{equation}\label{Intro: def of total length of branches that supports m leaves}
		L_{n,m}:=\sum_{i=1}^n L_{n,i,m}. 
	\end{equation}

	\begin{proposition}\label{Proposition: expected value, typical behavior}
		In the linear case with $2\le i \le n-m$, we have
		$$
		\E[L_{n,i,m}]=\frac{1}{n^2}.
		$$
	\end{proposition}
	If we sum over $2\le i \le n-m$, then we get 
	$$\E\left[\sum_{i=2}^{n-m} L_{n,i,m}\right] = \frac{n-m}{n^2},$$
	which means we get a triangular shape for the expected site frequency spectrum if we ignore $i=1$ and $i=n-m+1$.  For $i=1$ and $n-m+1$, the corresponding branch lengths have infinite mean and therefore $E[L_{n,m}]=\infty$. However, the next theorem says these branches do not have a major effect when we consider the typical behavior of the total branch length.

	\begin{theorem}\label{Theorem: main result}
		In both the linear and circular cases, let $m$ be a fixed positive integer. Let $L_{n,m}$ be defined as in \eqref{Intro: def of total length of branches that supports m leaves}. Then  $n L_{n,m}$ converges to 1 in probability as $n$ goes to infinity.
	\end{theorem}
	
	We will focus on the proof of Theorem \ref{Theorem: main result} in the linear case and then deduce the result for the circular case from the linear case. Throughout the rest of the paper, unless otherwise specified, $C=C_m$ will be some constant which may depend on $m$ and vary from line to line.

	\section{Results for Brownian motion}
	In this section we summarize some results about Brownian motion. For Lemmas \ref{Lemma: reflection principle} and \ref{Lemma: 1-dim BM}, we refer the reader to Sections 7.4 and 7.5 of \cite{durrett2019probability}.
	\begin{lemma}[Reflection Principle]\label{Lemma: reflection principle}
		For every $x>0$ and $t\ge 0$,
		$$\P\left(\max_{0\le s\le t}W^0_t \ge x\right)=\P(|W^0_t|\ge x)\le C\frac{\sqrt{t}}{x}\exp\left(-\frac{x^2}{2t}\right).$$
	\end{lemma}
	
	\begin{lemma} \label{Lemma: 1-dim BM}
		For $x\in \R$, let 
		$$
		T_{x}=\inf\{t\ge 0: W^0_t=x \},
		$$
		be the hitting time of $x$ for a 1-dimensional Brownian motion. Then for $a,b>0$, 
		$$
		\E[T_a\wedge T_{-b}]=ab.
		$$
	\end{lemma}
	
	In Example 1 of \cite{garbit2014exit}, Garbit and Raschel computed the asymptotics of the tail distribution of the exit time of a cone for a 2-dimensional Brownian motion:
	\begin{lemma}\label{Lemma: tail distribution for exit time from a cone}
		For each nonzero $(x,y)\in\R^2$, let $\phi(x,y)$ be the angle between $(x,y)$ and $(1,0)$, i.e. $\phi(x,y)=arccos(x/\sqrt{x^2+y^2})\in[0,\pi]$. For $\theta\in(0,2\pi]$, we define the cone $\mathcal{C}_{\theta}$ 
		$$
		\mathcal{C}_{\theta}:=\{(x,y)\in \R^2 :\phi(x,y)< \theta/2\}.
		$$
		For any $(a,b)\in \mathcal{C}_\theta$, let 
		$$
		T_{a,b} = \inf\left\{t\ge 0: \phi\left(W^a_t, W^b_t\right)=\theta/2\right\}
		$$
		be the exit time of $\mathcal{C}_{\theta}$ for the 2-dimensional Brownian motion starting from $(a,b)$. Then, using the notation $f(t)\sim g(t)$ to mean that $\lim_{t\rightarrow\infty} f(t)/g(t)=1$, we have
		$$
		\P(T_{a,b}>t)\sim C_{a,b} t^{-\pi/(2\theta)}, \qquad\text{ for some constant } C_{a,b}.
		$$
	\end{lemma}
	\begin{remark}\label{Remark: tail distribution for exit time from a cone}
		Using Brownian scaling, we have
		$$
		\P(T_{a/n,b/n}>t)\sim C_{a,b} (n^2 t)^{-\pi/(2\theta)}, \qquad\text{ for some constant } C_{a,b}.
		$$
	\end{remark}
	
	\begin{lemma}\label{Lemma: expected value of the external branch length, linear}
		For real numbers $x<y<z$, let $W^x$, $W^y$, and $W^z$ be independent one-dimensional Brownian motions starting from $x$, $y$, and $z$ respectively. Let $\tau_{\{x,y\}}:=\inf \{t\ge 0: W^x_t=W^y_t\}$ and $\tau_{\{y,z\}}:=\inf \{t\ge 0: W^y_t=W^z_t\}$. Then
		$$
		\E[\tau_{\{x,y\}}\wedge \tau_{\{y,z\}}] = (z-y)(y-x).
		$$
	\end{lemma}
	\begin{proof}
		Define 
		$$
		B_t := (B^1_t, B^2_t) = (W^y_t-W^x_t, W^z_t-W^y_t) \text{ for } t\ge 0,
		$$
		and 
		$$
		T_l := \tau_{\{x,y\}}\wedge \tau_{\{y,z\}}\wedge \inf\left\{t\ge 0:B^1_t+B^2_t=l\right \} \text{ for } l>z-x.
		$$
		Applying It\^{o}'s formula to $X_t = B^1_t B^2_t (B^1_t+B^2_t-l)$, and using the fact that $\langle B^1\rangle_t=\langle B^2\rangle_t=2t$, and $\langle B^{1}, B^{2}\rangle_t=-t$, we have
		\begin{equation*}
			\begin{split}
				X_t 
				&= X_0+ \int_0^t \left(2B^1_s B^{2}_s +(B^{2}_s)^2 - l B^{2}_s\right)\ dB^1_{s} +\int_0^t \left(2B^1_s B^{2}_s +(B^{1}_s)^2 - l B^{1}_s\right)\ dB^{2}_{s}\\
				&\qquad+\frac{1}{2}\int_0^t 2B^{2}_s \
				d\langle B^{1}\rangle_s+\frac{1}
				{2}\int_0^t 2B^{1}_s\ d\langle
				B^{2}\rangle_s + \int_0^t \left(2B^1_s +2B^2
				_s -l\right) \ d\langle B^1, B^2\rangle_s\\
				&=X_0 + \int_0^t \left(2B^1_s B^2_s +(B^2
				_s)^2 - l B^2_s\right)\ dB^1_{s} +\int_0^t
				\left(2B^1_s B^2_s +(B^1_s)^2 - l B^1_s\right)\
				dB^2_s\\
				&\qquad+\int_0^t 2B^2_s\ ds+\int_0^t
				2B^1_s\ ds- \int_0^t \left(2B^1_s +2B^2_s -l\right)
				\ ds\\
				&=X_0 + \int_0^t \left(2B^1_s B^2_s +(B^2
				_s)^2 - l B^2_s\right)\ dB^1_s +\int_0^t
				\left(2B^1_s B^2_s +(B^1_s)^2 - l B^1_s\right)\
				dB^2_s+lt.\\	
			\end{split}
		\end{equation*} 
		Since 
		$$
		\E\left[\int_0^t \left(2B^1_sB^{2}_s +(B^{2}_s)^2 - l B^{2}_s\right)^2 \ ds +\int_0^t \left(2B^1_sB^{2}_s +(B^{1}_s)^2 - l B^{1}_s\right)^2\ ds\right]<\infty \qquad \text{ for all }t\ge 0,
		$$
		the process
		$$
		\int_0^t \left(2B^1_sB^{2}_s +(B^{2}_s)^2 - l B^{2}_s\right)\ dB^1_{s} +\int_0^t \left(2B^1_sB^{2}_s +(B^{1}_s)^2 - l B^{1}_s\right)\ dB^{2}_{s},\qquad t\ge 0,
		$$
		is therefore a martingale. By stopping $(X_t)_{t\ge 0}$ at $t\wedge T_l$ and
		taking expectations, we have
		$$
		\E\left[X_{t\wedge T_l}\right]= \E[X_0]+l\E\left[t\wedge T_l\right].
		$$
		Since $\lim_{t\rightarrow\infty} X_{t\wedge T_l}
		=X_{T_l}=0$ a.s. and $|X_{t\wedge T_l}|\le l^3$,
		it follows from the bounded convergence theorem
		that
		$$
		\E[T_l]=-\frac{\E[X_0]}{l} = (y-x)(z-y)\cdot\frac{l-(z-x)}{l}.	
		$$
		Taking the limit as $l$ goes to infinity gives $\E[\tau_{\{x,y\}}\wedge \tau_{\{y,z\}}] = (z-y)(y-x)$, which concludes the proof.
	\end{proof}

	\section{A single branch}
	\subsection{The tail distribution}
	Recall from \eqref{Intro: linear, formula for length of the ith branch that supports m leaves} that in the linear case with $2\le i \le n-m$, we have
	$$
	L_{n,i,m}=\left(\tau_{n,{\{i-1,i\}}}\wedge\tau_{n,\{i+m-1,i+m\}}-\tau_{n,\{i,i+m-1\}}\right)^+\le \tau_{n,{\{i-1,i\}}}\wedge\tau_{n,\{i+m-1,i+m\}}\le \tau_{n,{\{i-1,i,i+m\}}}.
	$$
	We give a bound on the tail of the distribution of $\tau_{n,{\{i-1,i,i+m\}}}$.
	\begin{lemma}\label{Lemma: tail distribution for a single branch}
		In the linear case with $2\le i\le n-m$, there exists a constant $C=C_m$ such that
		\begin{equation}\label{A single branch: tail distribution}
			\P\left(\tau_{n,{\{i-1,i,i+m\}}}\ge t\right) \le C n^{-3}t^{-3/2}\qquad\text{ for all } t>0.
		\end{equation}
		In particular,
		$$
		\P\left(\tau_{n,{\{i-1,i\}}}\wedge\tau_{n,\{i+m-1,i+m\}}\ge t\right)\le C n^{-3}t^{-3/2}\qquad\text{ for all } t>0.
		$$
	\end{lemma}
	
	\begin{proof}
		We consider the 2-dimensional (correlated) Brownian motion
		\begin{equation*}
			B_t=(B^1_t, B^2_t) = \left(W^{i/n}_t-W^{(i-1)/n}_t, W^{(i+m)/n}_t-W^{i/n}_t\right), \qquad\text{ for } t\ge 0.
		\end{equation*}
		By Remark \ref{Remark: equivalence in distribution}, writing $\stackrel{d}{=}$ for equivalence in distribution, we have, 
		\begin{equation}\label{A single branch: tail distribution, equivalence in distribution}
			\begin{split}
				\tau_{n,\{i-1,i,i+m\}}&=\inf\left\{t\ge 0: \check{Z}^{e_{n,i-1}}_t=\check{Z}^{e_{n,i}}_t \text{ or }\check{Z}^{e_{n,i}}_t=\check{Z}^{e_{n,i+m}}_t\right\}\\
				&\overset{d}=\inf\left\{t\ge 0: {Z}^{e_{n,i-1}}_t={Z}^{e_{n,i}}_t \text{ or }{Z}^{e_{n,i}}_t={Z}^{e_{n,i+m}}_t\right\}\\
				&=\inf\left\{t\ge 0: W^{(i-1)/n}_t=W^{i/n}_t \text{ or }W^{i/n}_t=W^{(i+m)/n}_t \right \}\\
				&= \inf\left\{t\ge 0: B^1_t=0 \text{ or } B^2_t=0 \right \}.\\
			\end{split}
		\end{equation}
		Consider the linear transformation of $\R^2$ defined by
		$$
		\varphi(x,y) := \left( \sqrt{\frac{2}{3}} \left(x+\frac{1}{2}y\right), \sqrt{\frac{1}{2}} y\right).
		$$
		Note that the process
		$$
		\varphi(B^1_t, B^2_t)=\left(\sqrt{\frac{2}{3}}\left(\frac{1}{2}W^{(i+m)/n}_t+\frac{1}{2}W^{i/n}_t-W^{(i-1)/n}_t\right), \sqrt{\frac{1}{2}}\left(W^{(i+m)/n}_t-W^{i/n}_t\right)\right), \qquad\text{ for } t\ge 0,
		$$
		is a two dimensional Brownian motion with independent components and unit variance in each component, starting from
		$\varphi(B^1_0, B^2_0) = \varphi(1/n, m/n)$. Also, the image of the first quadrant under $\varphi$ is $\{(x,y)\in \R^2: x>0, 0<y<\sqrt{3}x\}$, a cone with angle $\pi/3$ up to a rotation. By Remark \ref{Remark: tail distribution for exit time from a cone}, there exists a constant $C=C_m$ such that
		$$
		\P\left(\tau_{n,\{i-1,i,i+m\}}\ge t\right)\le C n^{-3}t^{-3/2},\qquad\text{ for all } t>0,
		$$
		which proves \eqref{A single branch: tail distribution}.
	\end{proof}
	
	\subsection{The expected value}
	
	We now give the proof of Proposition \ref{Proposition: expected value, typical behavior}.
	\begin{proof}[Proof of Proposition \ref{Proposition: expected value, typical behavior}]
		Recall from \eqref{Intro: linear, formula for length of the ith branch that supports m leaves} that for $2\le i \le n-m$, we have
		$$
		L_{n,i,m}=\left(\tau_{n,{\{i-1,i\}}}\wedge\tau_{n,\{i+m-1,i+m\}}-\tau_{n,\{i,i+m-1\}}\right)^+.
		$$
		For $m=1$, since $\tau_{n,\{i\}}=0$ by convention, we have 
		$$
		L_{n,i,1}=\left(\tau_{n,{\{i-1,i\}}}\wedge\tau_{n,\{i,i+1\}}-\tau_{n,\{i\}}\right)^+=\tau_{n,{\{i-1,i\}}}\wedge\tau_{n,\{i,i+1\}}.
		$$
		Then we have $\E[L_{n,i,1}]=1/{n^2}$ by Lemma \ref{Lemma: expected value of the external branch length, linear} with $x=(i-1)/n$, $y=i/n$ and $z=(i+1)/n$.
		
		For $m>1$, we write $A_m=\{i-1,i,i+m-1,i+m\}$. If the first coalescent event in $A_m$ is the coalescence of $i-1$ and $i$ or the coalescence of $i+m-1$ and $i+m$, then $L_{n,i,m}=0$. Otherwise, the first coalescent event is the coalescence of $i$ and $i+m-1$. Starting from $\tau_{n,A_m}=\tau_{n,\{i,i+m-1\}}$, the positions of the particles $\check{Z}^{e_{n,i-1}}$, $\check{Z}^{e_{n,i}}=\check{Z}^{e_{n,i+m-1}}$, and $\check{Z}^{e_{n,i+m}}$ evolve as independent Brownian motions before the next coalescent event among them, which is the same dynamics as for the case $m=1$. By Lemma \ref{Lemma: expected value of the external branch length, linear} with $x=\check{Z}^{e_{n,i-1}}_{\tau_{n,A_m}}$, $y=\check{Z}^{e_{n,i}}_{\tau_{n,A_m}}$ and $z=\check{Z}^{e_{n,i+m-1}}_{\tau_{n,A_m}}$, and the observation that one of the two factors is zero if the indicator fails to hold in the second equality, we have
		\begin{equation*}
			\begin{split}
				\E\left[L_{n,i,m}|\sigma\left(\boldsymbol{Z}^{\boldsymbol{e_n}}_t:0\le t\le \tau_{n,A_m}\right)\right]&= (\check{Z}^{e_{n,i}}_{\tau_{n,A_m}}-\check{Z}^{e_{n,i-1}}_{\tau_{n,A_m}})(\check{Z}^{e_{n,i+m}}_{\tau_{n,A_m}}-\check{Z}^{e_{n,i+m-1}}_{\tau_{n,A_m}})\mathbbm{1}_{\tau_{n,A_m}=\tau_{n,\{i,i+m-1\}}}\\
				\qquad&=(\check{Z}^{e_{n,i}}_{\tau_{n,A_m}}-\check{Z}^{e_{n,i-1}}_{\tau_{n,A_m}})(\check{Z}^{e_{n,i+m}}_{\tau_{n,A_m}}-\check{Z}^{e_{n,i+m-1}}_{\tau_{n,A_m}}).
			\end{split}
		\end{equation*}
		Taking the expectation, we have
		$$
		\E\left[L_{n,i,m}\right] = \E\left[(\check{Z}^{e_{n,i}}_{\tau_{n,A_m}}-\check{Z}^{e_{n,i-1}}_{\tau_{n,A_m}})(\check{Z}^{e_{n,i+m}}_{\tau_{n,A_m}}-\check{Z}^{e_{n,i+m-1}}_{\tau_{n,A_m}})\right].
		$$
		In view of Remark \ref{Remark: equivalence in distribution} with $A=A_m$,  we consider the 3-dimensional Brownian motion 
		$$
		B_t = (B^1_t, B^2_t, B^3_t)=\left(W^{i/n}_t-W^{(i-1)/n}_t, W^{(i+m-1)/n}_t-W^{i/n}_t, W^{(i+m)/n}_t-W^{(i+m-1)/n}_t\right), \qquad\text{ for } t\ge 0,
		$$
		and define
		\begin{equation*}
			T := \inf\left\{t\ge 0: B^1_t=0 \text{ or } B^2_t=0 \text{ or } B^3_t=0\right \}.
		\end{equation*}
		Then we have
		$$
		\E[L_{n,i,m}]=\E\left[B^1_{T}B^3_T\right].
		$$
		Applying It\^{o}'s formula to $X_t = B^1_t B^3_t$, and using the fact that $B^1_t$ and $B^3_t$ are independent, we have
		\begin{equation*}
			\begin{split}
				X_t = X_0+\int_0^t B^3_s\ d B^1_s+\int_0^t B^1_s \ dB^3_s.
			\end{split}
		\end{equation*}
		Since 
		$$
		\E\left[\int_0^t \left(B^3_s\right)^2\ ds+\int_0^t \left(B^1_s\right)^2 \ ds\right]<\infty \qquad\text{ for all }t<\infty,
		$$
		the process $(X_t)_{t\ge 0}$ is therefore a martingale. Stopping $(X_t)_{t\ge 0}$ at $t\wedge T$ and taking expectations, we have
		$$
		\E[X_{t\wedge T}]=\E[X_0]=\frac{1}{n^2}.
		$$
		It remains to show that $\E[X_T]=\lim_{t\rightarrow\infty} \E[X_{t\wedge T}]$. We have
		\begin{equation}\label{A single branch: expected value, taking the limit}
			|\E[X_T]-E[X_{t\wedge T}]|\le E[|X_{T}|\mathbbm{1}_{T>t}]+E[|X_{t}|\mathbbm{1}_{T>t}].   
		\end{equation}
		Note that $B_t^1=W^{i/n}_t-W^{(i-1)/n}_t$ and $B^3_t=W^{(i+m)/n}_t-W^{(i+m-1)/n}_t$ are nonnegative for $t\le T$, so the process $(X_{T\wedge t})_{t\ge 0}$ is nonnegative. Applying Fatou's lemma, we have
		$$\E[X_T]\le \liminf_{t\rightarrow\infty} \E[X_{T\wedge t }]=\frac{1}{n^2}<\infty.$$
		Then, by the dominated convergence theorem, $\E[X_T\mathbbm{1}_{T>t}]$ goes to $0$ as $t$ goes to infinity. For $E[|X_{t}|\mathbbm{1}_{T>t}]$, since 
		$$
		T\le \inf\left\{t\ge 0: B^1_t=0 \text{ or } B^2_t+B^3_t=0 \right \}\overset{d}{=}\tau_{n,\{i-1,i,i+m\}},
		$$
		it follows from Lemma \ref{Lemma: tail distribution for a single branch} that there exists a constant $C=C_m$ such that
		$$
		\P(T>t)\le C n^{-3}t^{-3/2}\le C t^{-3/2}.
		$$
		With the constant $C$ fixed, we show that for any event $A$,
		\begin{equation}\label{A single branch: expected value, taking the limit, second term}
			\P(A)\le C t^{-3/2}\implies\E(|X_t| \mathbbm{1}_{A})\le  (3C/2)t^{-1/2}\log t +(C-C\log C)t^{-1/2}, 
		\end{equation}
		which proves that the second term of \eqref{A single branch: expected value, taking the limit} goes to 0 as $t$ goes to infinity. Since
		\begin{equation*}
			\P(|X_t|>x)=\P(|B^1_t B^3_t|>x)\le \P\left((B^1_t)^2+(B^3_t)^2>2x\right)=\exp\left(-\frac{x}{t}\right),
		\end{equation*}
		it follows that $|X_t|$ is stochastically dominated by a random variable $Y_t$ whose tail probability is $P(Y_t>x)=\exp(-x/t)$ for all $x\ge 0$. For the random variable $Y_t$, we choose $x_0$ such that $P(Y_t>x_0)=Ct^{-3/2}$, i.e. $x_0=(3t\log t)/2-t\log C$. We have 
		\begin{equation*}
			\begin{split}
				\E(Y_t \mathbbm{1}_{Y_t>x_0})&=\int_0^\infty \P(Y_t\mathbbm{1}_{\{Y_t>x_0\}}>x)\ dx\\
				&=\int_0^{x_0} \P(Y_t>x_0) \ dx+ \int_{x_0}^{\infty} \P(Y_t>x) \ dx \\
				&= x_0\P(Y_t>x_0)  +t\exp\left(-\frac{x_0}{t}\right)\\
				&= (3C/2)t^{-1/2}\log t +(C-C\log C)t^{-1/2},\\
			\end{split}
		\end{equation*}
		which proves \eqref{A single branch: expected value, taking the limit, second term} because $Y_t$ stochastically dominates $|X_t|$.

	\end{proof}

	\section{Law of Large Numbers}
	\subsection{External branch lengths in sub-systems}
	Now we consider the branch lengths in sub-systems. The reason for doing this is to exploit the independence of branch lengths in the sub-systems. Also, the lengths in the sub-systems agree with those in the whole system with sufficiently high probability. More precisely, let $d_{\mathbb{S}^1}(\cdot,\cdot)$ be the metric of $\mathbb{S}^1$ given by the arc length. For the coalescing Brownian motion with $n$ particles and $i,j\in\{1,2,\dots, n\}$, we define 
	$$
	d_n(i,j):=\left\{
	\begin{array}{cc}
		|e_{n,i}-e_{n,j}|=|i-j|/n, & \text{ if } E=\R,\\
		\frac{1}{2\pi}d_{\mathbb{S}^1}(e_{n,i}, e_{n,j}),& \text{ if } E=\mathbb{S}^1.
	\end{array}
	\right.
	$$
	Fix some $\epsilon>0$ sufficiently small, for example, $\epsilon=0.01$. We define the neighborhood of $i$ as
	$$
	N_n(i):=\left\{j\in\{1,\dots, n\}: d_n(i,j)\le n^{-2/3+\epsilon}\right\}.
	$$
	The coalescing Brownian motion in $N_n(i)$ is obtained by considering only $\boldsymbol{Z^{e_{N_n(i)}}} = (Z^{e_{n,k}})_{k\in N_{n}(i)}$. Quantities in this system are subscripted with $N_{n}(i)$ instead of $n$. For example, the length of the portion of the $i$th branch that supports $m$ leaves in this system is denoted by $L_{N_{n}(i),i,m}$. We can recover the coalescing Brownian motion with $n$ particles from the coalescing Brownian motion in $N_n(i)$ by taking the Brownian motions $(Z^{e_{n,j}})_{j\notin N_n(i)}$ into account. Note that the distribution of the positions of particles in $N_n(i)$ remains invariant, i.e.
	$$
	\left(\check{Z}^{e_{N_n(i), k}}\right)_{k\in N_n(i)}\stackrel{d}{=}\left(\check{Z}^{e_{n, k}}\right)_{k\in N_n(i)}.
	$$
	In particular, we have $L_{N_{n}(i),i,m}\stackrel{d}{=} L_{n,i,m}$.
	
	\begin{lemma}\label{Lemma: coupling, linear subclone and linear population}
		Let $L_{N_{n}(i),i}$ be defined as above. In the linear case with $2\le i\le n-m$, we have
		\begin{enumerate}
			\item $\boldsymbol{{Z^{e_{N_n(i)}}}}$ and $\boldsymbol{{Z^{e_{N_n(j)}}}}$ are independent if $N_{n}(i)\cap N_{n}(j)=\emptyset$.
			\item For all $\epsilon>0$, there exists a constant $C=C_\epsilon$ such that $\P(L_{N_{n}(i),i,m}\neq L_{n,i,m})\le C n^{-1-\epsilon/2}$.
		\end{enumerate}
	\end{lemma}
	
	\begin{proof}
		The first claim is straightforward. For the second claim, we define
		$$
		\underline{i}:=\inf\{j\in N_{n}(i): \tau_{N_n (i), \{i,j\}}\le \tau_{N_n(i), \{i-1,i \}}\wedge\tau_{N_n(i), \{i+m-1,i+m \}}\},\qquad
		$$
		the leftmost particle in $N_n(i)$ that coalesces with the $i$th particle before time $\tau_{N_n(i), \{i-1,i \}}\wedge\tau_{N_n(i), \{i+m-1,i+m \}}$.
		Then $L_{n,i,m}\neq L_{N_{n}(i),i,m}$ only if $\check{Z}^{e_{\underline{i}}}$ 
		coalesces with $Z^{e_j}$, for some $j\notin N_{n}(i)$ before $\tau_{N_n(i), \{i-1,i \}}\wedge\tau_{N_n(i), \{i+m-1,i+m \}}$, and we have
		
		\begin{align}\label{LLN: LLN: coupling, linear subclone and linear population, left part}
			&\P(L_{n,i,m}\neq L_{N_n (i), i, m})\nonumber\\
			&\qquad\le \P\left(\exists j,t:\ j\notin N_{n}(i),j<i,\  t<\tau_{N_n(i), \{i-1,i \}}\wedge\tau_{N_n(i), \{i+m-1,i+m \}},\  \check{Z}^{e_{n,\underline{i}}}_t = Z^{e_{n,j}}_t\right).    
		\end{align}
		We now bound the right hand side of \eqref{LLN: LLN: coupling, linear subclone and linear population, left part}. Let $i_0=\lfloor i- n^{1/3+\epsilon}/2\rfloor$, and let $j\notin{N_n(i)}, j<i$. If $\underline{i}>i_0$, then $Z^{e_{n,j}}_t$ coalesces with $\check{Z}^{e_{n,i_0}}$ before $Z^{e_{n,j}}_t$ coalesces with $\check{Z}^{e_{n,\underline{i}}}$. It follows that
		\begin{align}\label{LLN: coupling, linear subclone and linear population, split}
			&\P\left(\exists j,t:\ j\notin N_{n}(i),j<i,\  t<\tau_{N_n(i),\{i-1,i\}}\wedge\tau_{N_n(i),\{i+m-1,i+m\}},\  \check{Z}^{e_{n,\underline{i}}}_t = Z^{e_{n,j}}_t\right)\nonumber\\
			&\qquad\le \P(\tau_{N_n(i),\{i-1,i\}}\wedge\tau_{N_n(i),\{i+m-1,i+m\}}\ge n^{-4/3+\epsilon})\nonumber\\
			&\qquad\qquad+\P(\exists j,t:\ j\notin N_{n}(i),j<i,t\le n^{-4/3+\epsilon}, \check{Z}^{e_{n,\underline{i}}}_t = Z^{e_{n,j}}_t)\nonumber\\
			&\qquad\le \P(\tau_{N_n(i),\{i-1,i\}}\wedge\tau_{N_n(i),\{i+m-1,i+m\}}\ge n^{-4/3+\epsilon})\nonumber\\
			&\qquad\qquad+\P(\exists j,t:\ j\notin N_{n}(i),j<i,t\le n^{-4/3+\epsilon},\check{Z}^{e_{n,i_0}}_t = Z^{e_{n,j}}_t)+\P\left(\underline{i}\le i_0\right)\nonumber\\
			&\qquad \le \P(\tau_{N_n(i),\{i-1,i\}}\wedge\tau_{N_n(i),\{i+m-1,i+m\}}\ge n^{-4/3+\epsilon})\nonumber\\
			&\qquad\qquad+\sum_{j\notin N_n(i),j<i}\P(\exists t\le n^{-4/3+\epsilon}: \check{Z}^{e_{n,i_0}}_t = Z^{e_{n,j}}_t)+\P\left(\underline{i}\le i_0\right).
		\end{align}	
		By Lemma \ref{Lemma: tail distribution for a single branch}, we have
		\begin{align}\label{LLN: coupling, linear subclone and linear population, long coalescent time}
			\P(\tau_{N_n(i),\{i-1,i\}}\wedge\tau_{N_n(i),\{i+m-1,i+m\}}\ge n^{-4/3+\epsilon})&=\P(\tau_{n,\{i-1,i\}}\wedge\tau_{n,\{i+m-1,i+m\}}\ge n^{-4/3+\epsilon})\nonumber\\
			&\le C n^{-1-3\epsilon/2}.   
		\end{align}
		For $j\notin N_{n}(i)$, by Lemma \ref{Lemma: reflection principle}, we have
		\begin{align}\label{LLN: coupling, linear subclone and linear population, short coalescent time}
			\P\left(\exists t\le n^{-4/3+\epsilon}: \check{Z}^{e_{n,i_0}}_t = Z^{e_{n,j}}_t\right)
			&=\P\left(\exists t\le n^{-4/3+\epsilon} :\sqrt{2}W^0_t = d_n(i_0,j)\right)\nonumber\\
			&\le \P\left(\exists t\le n^{-4/3+\epsilon} :W^0_t = {n^{-2/3+\epsilon}}/\sqrt{2}\right)\nonumber\\
			&=\P\left(|W^0_{n^{-4/3+\epsilon}}| \ge {n^{-2/3+\epsilon}}/\sqrt{2}\right)\nonumber\\
			&\le Cn^{-\epsilon/2}\exp(-n^{\epsilon}/4)\nonumber\\
			&\le Cn^{-2-3\epsilon/2}.
		\end{align}
		Also, by \eqref{LLN: coupling, linear subclone and linear population, long coalescent time} and the proof of \eqref{LLN: coupling, linear subclone and linear population, short coalescent time}, we have
		\begin{align}\label{LLN: coupling, linear subclone and linear population, long range}
			\P(\underline{i}\le i_0)&=\P(\tau_{N_n(i),\{i_0, i\}}\le \tau_{N_n(i),\{i-1,i\}}\wedge\tau_{N_n (i),\{i+m-1,i+m\}})\nonumber\\
			&\le \P(\tau_{N_n(i),\{i-1,i\}}\wedge\tau_{N_n(i),\{i+m-1,i+m\}}\ge n^{-4/3+\epsilon})+\P(\tau_{N_n(i),\{i_0, i\}}\le n^{-4/3+\epsilon})\nonumber\\
			&\le C n^{-1-3\epsilon/2}+C n^{-2-3\epsilon/2}\nonumber\\
			&\le C n^{-1-3\epsilon/2}.
		\end{align}
		Combining \eqref{LLN: coupling, linear subclone and linear population, split}, \eqref{LLN: coupling, linear subclone and linear population, long coalescent time}, \eqref{LLN: coupling, linear subclone and linear population, short coalescent time} and \eqref{LLN: coupling, linear subclone and linear population, long range} gives us
		$$
		\P\left(\exists j,t:\ j\notin N_n(i),j<i,\  t<L_{N_n(i), \{i-1,i\}},\  \check{Z}^{e_{n,\underline{i}}}_t = Z^{e_{n,j}}_t\right)\le Cn^{-1-3\epsilon/2},
		$$
		which completes the proof.
	\end{proof}
	
	\subsection{Proof of the main results in the linear case}
	We now proceed to the proof of Theorem \ref{Theorem: main result} in the linear case.
	\begin{proof} 
		Recall from \eqref{Intro: linear, formula for length of the ith branch that supports m leaves} that 
		$$
		L_{n,1,m}=\left(\tau_{n,\{m,m+1\}}-\tau_{n,\{1,m\}}\right)^{+}\le \tau_{n,\{m,m+1\}}.
		$$
		Since $\tau_{n,\{m,m+1\}}$ is the hitting time of two Brownian particles which begin a distance $1/n$ apart, it follows that $n\tau_{n,\{m,m+1\}}$ converges to 0 in probability. Hence, $nL_{n,1,m}$ converges to 0 in probability. Similarly, $nL_{n,n-m+1,m}$ converges to 0 in probability. Therefore, it suffices to show that
		\begin{equation}\label{LLN: edges effects ignored}
			n\sum_{i=2}^{n-m}L_{n,i,m}\stackrel{\P}{\longrightarrow}1, 
		\end{equation}
		where $\stackrel{\P}{\longrightarrow}$ denotes convergence in probability.
		By Lemma \ref{Lemma: coupling, linear subclone and linear population}, using the union bound, we have
		$$
		\P\left(\sum_{i=2}^{n-m} L_{n,i,m} \neq \sum_{i=2}^{n-m} L_{N_n(i),i,m}\right)\le Cn^{-1-\epsilon/2}\cdot n =C n^{-\epsilon/2}.
		$$
		Define a truncated versions of $L_{n,i,m}$ as 
		$$ 
		\widetilde{L}_{n,i,m} := \left(\tau_{n,\{i-1,i\}}\wedge\tau_{n,\{i+m-1,i+m\}}\wedge n^{-4/3+\epsilon}-\tau_{n(i),\{i,i-m+1\}}\right)^+,
		$$
		and of $L_{N_n(i),i,m}$ as
		$$ 
		\widetilde{L}_{N_n(i),i,m} := \left(\tau_{N_n(i),\{i-1,i\}}\wedge\tau_{N_n(i),\{i+m-1,i+m\}}\wedge n^{-4/3+\epsilon}-\tau_{N_{n}(i),\{i,i-m+1\}}\right)^+.
		$$
		By Lemma \ref{Lemma: tail distribution for a single branch}, we have
		\begin{equation*}
			\begin{split}
				\P\left(\sum_{i=2}^{n-m} L_{N_n(i),i,m}\neq \sum_{i=2}^{n-m} \widetilde{L}_{N_n(i),i,m}\right)
				&\le \sum_{i=2}^{n-m}\P(L_{N_n(i),i,m}\neq\widetilde{L}_{N_n(i),i,m})\\
				&\le \sum_{i=2}^{n-m}\P\left(\tau_{N_{n}(i),\{i-1,i\}}\wedge \tau_{N_{n}(i),\{i+m-1,i+m\}}>n^{-4/3+\epsilon}\right)\\
				&= \sum_{i=2}^{n-m}\P\left(\tau_{n,\{i-1,i\}}\wedge \tau_{n,\{i+m-1,i+m\}}>n^{-4/3+\epsilon}\right)\\
				&\le  Cn\cdot n^{-1-3\epsilon/2}\\
				&=C n^{-3\epsilon/2}.
			\end{split}
		\end{equation*}
		Therefore, it suffices to show that
		\begin{equation}\label{LLN: truncated locally dependent sequence}
			n \sum_{i=2}^{n-m} \widetilde{L}_{N_n(i),i,m}\stackrel{\P}{\longrightarrow}1.
		\end{equation}
		
		For the expected value of \eqref{LLN: truncated locally dependent sequence}, by Proposition \ref{Proposition: expected value, typical behavior}, we have
		$$
		\E[L_{N_n(i),i,m}]=\E[L_{n,i,m}]=\frac{1}{n^2},
		$$
		and by Lemma \ref{Lemma: tail distribution for a single branch}, we have
		\begin{equation*}
			\begin{split}
				\E\left[L_{N_n(i),i,m}-\widetilde{L}_{N_n(i),i,m}\right]
				&\le \E\left[\left(\tau_{N_n(i),\{i-1,i\}}\wedge\tau_{N_{n}(i),\{i+m-1,i+m\}}-n^{-4/3+\epsilon}\right)^{+}\right]\\
				&= \E\left[\left(\tau_{n,\{i-1,i\}}\wedge\tau_{n,\{i+m-1,i+m\}}-n^{-4/3+\epsilon}\right)^{+}\right]\\
				&= \int_{n^{-4/3+\epsilon}}^\infty \P\left(\tau_{n,\{i-1,i\}}\wedge\tau_{n,\{i+m-1,i+m\}}>t\right)\ dt\\
				&\le C \int_{n^{-4/3+\epsilon}}^\infty n^{-3}t^{-3/2}\ dt\\
				&= C n^{-7/3-\epsilon/2}.
			\end{split}
		\end{equation*}
		It follows that
		\begin{equation}\label{LLN: difference of the expected value of tau n,i-1,i,i+1 and its truncation}
			\lim_{n\rightarrow\infty}\E\left[n \sum_{i=2}^{n-m} \widetilde{L}_{N_n(i),i,m}\right]=1.
		\end{equation}

		For the variance of \eqref{LLN: truncated locally dependent sequence}, by the first claim of Lemma \ref{Lemma: coupling, linear subclone and linear population}, $\widetilde{L}_{N_n(i),i,m}$ and $\widetilde{L}_{N_n(i),j,m}$ are independent if $N_{n}(i)$ and $N_{n}(j)$ are disjoint. Therefore, we have
		\begin{equation}\label{LLN: variance bound}
			\begin{split}
				\Var\left(n\sum_{i=2}^{n-m} \widetilde{L}_{N_n(i),i,m}\right)
				&=n^2\sum_{i=2}^{n-m}\sum_{j: N_{n}(j)\cap N_{n}(i)\neq\emptyset} \Cov\left(\widetilde{L}_{N_n(i),i,m},\widetilde{L}_{N_n(j),j,m}\right)\\
				&\le n^2\sum_{i=2}^{n-m}\sum_{j: N_{n}(j)\cap N_{n}(i)\neq\emptyset} \E\left[\left(\widetilde{L}_{N_{n}(i),i,m}\right)^2\right]\\
				&= n^2\sum_{i=2}^{n-m}\sum_{j: N_{n}(j)\cap N_{n}(i)\neq\emptyset} \E\left[\left(\widetilde{L}_{n,i,m}\right)^2\right]\\
				&\le Cn^2\cdot n\cdot n^{1/3+\epsilon}\E\left[\left(\widetilde{L}_{n,i,m}\right)^2\right].\\
			\end{split}
		\end{equation}
		Using Lemma \ref{Lemma: tail distribution for a single branch}, we have,
		\begin{equation*}
			\begin{split}
				\E\left[\left(\widetilde{L}_{n,i,m}\right)^2\right]
				&=\int_0^\infty 2t\P\left(\widetilde{L}_{n,i,m}\ge t\right)\ dt\\
				&=\int_0^{n^{-4/3+\epsilon}} 2t\P\left(L_{n,i,m}\ge t\right)\ dt\\
				&\le\int_0^{n^{-4/3+\epsilon}} 2t\P\left(\tau_{n,\{i-1,i\}}\wedge\tau_{n,\{i+m-1, i+m\}}\ge t\right)\ dt\\
				&\le\int_0^{n^{-4/3+\epsilon}} 2 t\cdot  C n^{-3} t^{-3/2}\ dt\\
				&\le C n^{-11/3+\epsilon/2}.
			\end{split}
		\end{equation*}
		Combining this with \eqref{LLN: variance bound}, we have
		$$
		\Var\left(n\sum_{i=2}^{n-m} \widetilde{L}_{N_n(i),i,m}\right)\le C n^2\cdot n \cdot n^{1/3+\epsilon}\cdot n^{-11/3+\epsilon/2}= C n^{-1/3+3\epsilon/2}.
		$$
		It follows from Chebyshev's inequality that
		\begin{equation*}
			n\sum_{i=2}^{n-m}  \widetilde{L}_{N_n(i),i,m}-\E\left[n\sum_{i=2}^{n-m}  \widetilde{L}_{N_n(i),i,m}\right]\stackrel{\P}{\longrightarrow}0.
		\end{equation*}
		Combining this with \eqref{LLN: difference of the expected value of tau n,i-1,i,i+1 and its truncation} gives the result.
	\end{proof}

	\begin{remark}\label{Remark: edges effects ignored}
		The same argument used to prove \eqref{LLN: edges effects ignored} implies that
		$$
		n\sum_{i=\lfloor n^{1/3+\epsilon} \rfloor }^{\lceil n-n^{1/3+\epsilon}\rceil}L_{n,i,m}\stackrel{\P}{\longrightarrow}1, 
		$$
		which we will use in the proof in the circular case.
	\end{remark}
	
	\subsection{Proof of the main result in the circular case}
	Now, we deduce the result in the circular case from the linear case. In this section, quantities in the circular case are subscript with $\mathbb{S}^1$. For example, the length of the portion of the $i$th branch that supports $m$ leaves is denoted by $L_{\mathbb{S}^1, n,i,m}$. The following lemma bounds the probability that the branch lengths differ in the linear and circular cases.
	\begin{lemma}\label{Lemma: coupling, linear population and circular population}
		Consider $n^{1/3+\epsilon}< i< n-n^{1/3+\epsilon}$ so that the neighborhood of $N_n(i)$ of $i$ consists of the same particles in the linear and circular case, we have
		\begin{equation}\label{LLN: coupling, linear subclone and circular subclone}
			\P\left(L_{N_n(i),i,m}\neq{L}_{\mathbb{S}^1,N_{n}(i),i,m}\right)\le C n^{-1-3\epsilon/2} 
		\end{equation}
		and 
		\begin{equation}\label{LLN: coupling, ciruclar subclone and circular population}
			\P\left({L}_{\mathbb{S}^1,N_{n}(i),i,m}\neq {L}_{\mathbb{S}^1,n,i,m}\right)\le C n^{-1-3\epsilon/2}.
		\end{equation}
	\end{lemma}
	
	\begin{proof}
		For the proof of \eqref{LLN: coupling, linear subclone and circular subclone}, let $i'$ and $i''$ be the smallest and largest indices in $N_{n}(i)$ respectively. Note that $|i''-i'|\le C n^{1/3+\epsilon}$. Suppose $\check{Z}^{e_{n,i''}}_t-\check{Z}^{e_{n,i'}}_t<1$ for all $t\le \tau_{N_n(i), \{i-1,i\}}\wedge \tau_{N_n(i), \{i+m-1,i+m\}}$. Then the $i'$th particle and the $i''$th particle do not coalesce, and we have
		$$
		L_{N_n(i),i,m}={L}_{\mathbb{S}^1,N_{n}(i),i,m}
		$$ 
		and
		$$\tau_{N_{n}(i),\{i-1,i\}}\wedge\tau_{N_{n}(i),\{i+m-1,i+m\}}={\tau}_{\mathbb{S}^1,N_{n}(i),\{i-1,i\}}\wedge{\tau}_{\mathbb{S}^1,N_{n}(i),\{i+m-1,i+m\}}.
		$$
		Using Lemma \ref{Lemma: reflection principle} in the fourth line and Lemma \ref{Lemma: tail distribution for a single branch} in the last line, we have
		\begin{equation*}
			\begin{split}
				&\P\left(L_{N_n(i),i,m}\neq{L}_{\mathbb{S}^1,N_{n}(i),i,m}\right)\\
				&\qquad\le \P\left(\exists t\le \tau_{N_{n}(i),\{i-1,i\}}\wedge\tau_{N_{n}(i),\{i+m-1,i+m\}}:\check{Z}^{e_{n,i'}}_t=\check{Z}^{e_{n,i''}}_t-1\right)\\
				&\qquad\le\P\left( \tau_{N_{n}(i),\{i-1,i\}}\wedge\tau_{N_{n}(i),\{i+m-1,i+m\}}>n^{-4/3+\epsilon}\right) +\P\left(\exists t\le n^{-4/3+\epsilon}: \check{Z}^{e_{n,i'}}_t=\check{Z}^{e_{n,i''}}_t-1\right)\\
				&\qquad=\P\left( \tau_{n,\{i-1,i\}}\wedge\tau_{n,\{i+m-1,i+m\}}>n^{-4/3+\epsilon}\right) +\P\left(\sqrt{2}|W^0_{n^{-4/3+\epsilon}}| \ge 1-\frac{|i''-i'|}{n} \right)\\
				&\qquad\le C n^{-1-3\epsilon/2},\\
			\end{split}
		\end{equation*}
		which proves \eqref{LLN: coupling, linear subclone and circular subclone}. 
		The same reasoning gives the next equation, which we will use in the proof of \eqref{LLN: coupling, ciruclar subclone and circular population}:
		\begin{equation}\label{LLN: coupling, circular subclone and circular population, long coalescent time, preparation}
			\P\left(\tau_{N_{n}(i),\{i-1,i\}}\wedge\tau_{N_{n}(i),\{i+m-1,i+m\}}\neq{\tau}_{\mathbb{S}^1,N_{n}(i),\{i-1,i\}}\wedge{\tau}_{\mathbb{S}^1,N_{n}(i),\{i+m-1,i+m\}}\right)\le C n^{-1-3\epsilon/2}.\\
		\end{equation}
		
		The proof of \eqref{LLN: coupling, ciruclar subclone and circular population} is similar to the proof of Lemma \ref{Lemma: coupling, linear subclone and linear population}. We define
		$$
		\underline{i}:=\inf\{j\in N_{n}(i): {\tau}_{\mathbb{S}^1,N_n (i), \{i,j\}}\le {\tau}_{\mathbb{S}^1,N_{n}(i),\{i-1,i\}}\wedge{\tau}_{\mathbb{S}^1,N_{n}(i),\{i+m-1,i+m\}}\},
		$$
		and 
		$$
		\overline{i}:=\sup\{j\in N_{n}(i): {\tau}_{\mathbb{S}^1,N_n (i), \{i+m,j\}}\le {\tau}_{\mathbb{S}^1,N_{n}(i),\{i-1,i\}}\wedge{\tau}_{\mathbb{S}^1,N_{n}(i),\{i+m-1,i+m\}}\}.
		$$
		Then
		${L}_{\mathbb{S}^1,n,i,m}\neq {L}_{\mathbb{S}^1,N_{n}(i),i,m}$ only if ${\check{Z}}^{e_{\mathbb{S}^1,n,\underline{i}}}$ or ${\check{Z}}^{e_{\mathbb{S}^1,n,\overline{i}}}$ coalesces with ${Z}^{e_{\mathbb{S}^1,n,j}}$, for some $j\notin N_{n}(i)$, before time ${\tau}_{\mathbb{S}^1,N_{n}(i),\{i-1,i\}}\wedge{\tau}_{\mathbb{S}^1,N_{n}(i),\{i+m-1,i+m\}}$. Therefore, we have
		\begin{align*}
			&\P\left({L}_{\mathbb{S}^1,N_{n}(i),i,m}\neq {L}_{\mathbb{S}^1,n,i,m}\right)\\
			&\qquad\le\P\left(\exists j,t:\ j\notin N_{n}(i),\  t<{\tau}_{\mathbb{S}^1,N_n(i),\{i-1,i\}}\wedge{\tau}_{\mathbb{S}^1,N_n(i),\{i+m-1,i+m\}},\  {\check{Z}}^{e_{\mathbb{S}^1,n,\underline{i}}}_t = {Z}^{e_{\mathbb{S}^1,n,j}}_t\right)\\ 
			&\qquad\qquad+\P\left(\exists j,t:\ j\notin N_{n}(i),\  t<{\tau}_{\mathbb{S}^1,N_n(i),\{i-1,i\}}\wedge{\tau}_{\mathbb{S}^1,N_n(i),\{i+m-1,i+m\}},\  {\check{Z}}^{e_{\mathbb{S}^1,n,\overline{i}}}_t = {Z}^{e_{\mathbb{S}^1,n,j}}_t\right)
		\end{align*}
		We bound the first term on the right hand side and the same argument can be applied to the second term. Writing $i_0=\left\lfloor i-n^{1/3+\epsilon}/2\right\rfloor$, we have
		\begin{align}\label{LLN: coupling, circular subclone and circular population, split}
			&\P\left(\exists j,t:\ j\notin N_{n}(i),\  t<{\tau}_{\mathbb{S}^1,N_n(i),\{i-1,i\}}\wedge{\tau}_{\mathbb{S}^1,N_n(i),\{i+m-1,i+m\}},\  {\check{Z}}^{e_{\mathbb{S}^1,n,\underline{i}}}_t = {Z}^{e_{\mathbb{S}^1,n,j}}_t\right)\nonumber\\
			&\qquad\le \P({\tau}_{\mathbb{S}^1,N_n(i),\{i-1,i\}}\wedge{\tau}_{\mathbb{S}^1,N_n(i),\{i+m-1,i+m\}}\ge n^{-4/3+\epsilon})\nonumber\\
			&\qquad\qquad+\P(\exists j,t:\ j\notin N_{n}(i),t\le n^{-4/3+\epsilon}, {\check{Z}}^{e_{\mathbb{S}^1,n,\underline{i}}}_t = {Z}^{e_{\mathbb{S}^1,n,j}}_t)\nonumber\\
			&\qquad\le \P({\tau}_{\mathbb{S}^1,N_n(i),\{i-1,i\}}\wedge{\tau}_{\mathbb{S}^1,N_n(i),\{i+m-1,i+m\}}\ge n^{-4/3+\epsilon})\nonumber\\
			&\qquad\qquad+\P(\exists j,t:\ j\notin N_{n}(i),t\le n^{-4/3+\epsilon},{\check{Z}}^{e_{\mathbb{S}^1,n,i_0}}_t = {Z}^{e_{\mathbb{S}^1,n,j}}_t)+\P\left(\underline{i}\le i_0\right)\nonumber\\
			&\qquad \le \P({\tau}_{\mathbb{S}^1,N_n(i),\{i-1,i\}}\wedge{\tau}_{\mathbb{S}^1,N_n(i),\{i+m-1,i+m\}}\ge n^{-4/3+\epsilon})\nonumber\\
			&\qquad\qquad+\sum_{j\notin N_n(i)}\P(\exists t\le n^{-4/3+\epsilon}: {\check{Z}}^{e_{\mathbb{S}^1,n,i_0}}_t = {Z}^{e_{\mathbb{S}^1,n,j}}_t)+\P\left(\underline{i}\le i_0\right).
		\end{align}	
		To bound the first term on the right hand side of \eqref{LLN: coupling, circular subclone and circular population, split}, we use \eqref{LLN: coupling, circular subclone and circular population, long coalescent time, preparation} to get
		\begin{align}\label{LLN: coupling, circular subclone and circular population, long coalescent time}
			&\P({\tau}_{\mathbb{S}^1,N_n (i),\{i-1,i\}}\wedge{\tau}_{\mathbb{S}^1,N_n (i),\{i+m-1,i+m\}}> n^{-4/3+\epsilon})\nonumber\\
			&\qquad\le\P\left(\tau_{N_{n}(i),\{i-1,i\}}\wedge\tau_{N_{n}(i),\{i+m-1,i+m\}}\neq{\tau}_{\mathbb{S}^1,N_{n}(i),\{i-1,i\}}\wedge{\tau}_{\mathbb{S}^1,N_{n}(i),\{i+m-1,i+m\}}\right)\nonumber\\
			&\qquad\qquad+\P({\tau}_{N_n(i),\{i-1,i\}}\wedge{\tau}_{N_n(i),\{i+m-1,i+m\}}\ge n^{-4/3+\epsilon})\nonumber\\
			&\qquad\le\P\left(\tau_{N_{n}(i),\{i-1,i\}}\wedge\tau_{N_{n}(i),\{i+m-1,i+m\}}\neq{\tau}_{\mathbb{S}^1,N_{n}(i),\{i-1,i\}}\wedge{\tau}_{\mathbb{S}^1,N_{n}(i),\{i+m-1,i+m\}}\right)\nonumber\\
			&\qquad\qquad+\P({\tau}_{n,\{i-1,i\}}\wedge{\tau}_{n,\{i+m-1,i+m\}}\ge n^{-4/3+\epsilon})\nonumber\\
			&\qquad\le C n^{-1-3\epsilon/2}.
		\end{align}
		To bound the second term on the right hand side of \eqref{LLN: coupling, circular subclone and circular population, split}, note that for two Brownian particles on the circle, it is more likely that they will coalesce through the smaller arc. Therefore, the same argument used to prove \eqref{LLN: coupling, linear subclone and linear population, short coalescent time} gives 
		\begin{align}\label{LLN: coupling, circular subclone and circular population, short coalescent time}
			\P(\exists t\le n^{-4/3+\epsilon}: {\check{Z}}^{e_{\mathbb{S}^1,n,i_0}}_t = {Z}^{e_{\mathbb{S}^1,n,j}}_t)&=\P(\exists t\le n^{-4/3+\epsilon}: {Z}^{e_{\mathbb{S}^1,n,i_0}}_t = {Z}^{e_{\mathbb{S}^1,n,j}}_t)\nonumber\\
			&\le 2\P(\exists t\le n^{-4/3+\epsilon}: Z^{e_{n,i_0}}_t = Z^{e_{n,j}}_t)\nonumber\\
			&\le Cn^{-2-3\epsilon/2}.
		\end{align}
		To bound the last term on the right hand side of \eqref{LLN: coupling, circular subclone and circular population, split}, again, since two Brownian particles on the circle are more likely to coalesce through the smaller arc, using \eqref{LLN: coupling, circular subclone and circular population, long coalescent time} in the third line and the proof of \eqref{LLN: coupling, circular subclone and circular population, short coalescent time} in the last line, we have
		\begin{equation}\label{LLN: coupling, circular subclone and circular population, long range}
			\begin{split}
				\P(\underline{i}\le i_0)&\le 2\P({\tau}_{\mathbb{S}^1,N_n(i),\{i_0, i\}}\le {\tau}_{\mathbb{S}^1,N_n(i),\{i-1,i\}}\wedge{\tau}_{\mathbb{S}^1,N_n (i),\{i+m-1,i+m\}})\\
				&\le 2\P({\tau}_{\mathbb{S}^1,N_n(i),\{i-1,i\}}\wedge{\tau}_{\mathbb{S}^1,N_n(i),\{i+m-1,i+m\}}\ge n^{-4/3+\epsilon})+2\P({\tau}_{\mathbb{S}^1,N_n(i),\{i_0, i\}}\le n^{-4/3+\epsilon})\\
				&\le C n^{-1-3\epsilon/2}+2\P(\tau_{N_n(i),\{i_0, i\}}\le n^{-4/3+\epsilon})\\
				&\le C n^{-1-3\epsilon/2}.
			\end{split}
		\end{equation}
		Then equation \eqref{LLN: coupling, ciruclar subclone and circular population} follows from \eqref{LLN: coupling, circular subclone and circular population, split}, \eqref{LLN: coupling, circular subclone and circular population, long coalescent time}, \eqref{LLN: coupling, circular subclone and circular population, short coalescent time} and \eqref{LLN: coupling, circular subclone and circular population, long range}.
	\end{proof}
	
	We now give the proof for Theorem \ref{Theorem: main result} in the circular case.
	\begin{proof}
		By Lemmas \ref{Lemma: coupling, linear subclone and linear population} and \ref{Lemma: coupling, linear population and circular population}, we have
		$$
		\P\left(\sum_{i=\lfloor n^{1/3+\epsilon}\rfloor}^{\lceil n-n^{1/3+\epsilon}\rceil} {L}_{\mathbb{S}^1,n,i,m}\neq \sum_{i=\lfloor n^{1/3+\epsilon}\rfloor}^{\lceil n-n^{1/3+\epsilon}\rceil} L_{n,i,m}\right)\le Cn^{-3\epsilon/2}.
		$$
		By Remark \ref{Remark: edges effects ignored}, it suffices to show that 
		$$
		n\left(\sum_{i=1}^{\lfloor n^{1/3+\epsilon}\rfloor-1}{L}_{\mathbb{S}^1,n,i,m}+\sum_{i=\lceil n-n^{1/3+\epsilon}\rceil+1}^{n}{L}_{\mathbb{S}^1,n,i,m}\right )\stackrel{\P}{\longrightarrow}0.
		$$
		Since ${L}_{\mathbb{S}^1,n,i,m}$ has the same distribution for all $1\le i\le n$, by Lemmas \ref{Lemma: coupling, linear subclone and linear population}, \ref{Lemma: coupling, linear population and circular population} and Proposition~ \ref{Proposition: expected value, typical behavior}, there exist random variables $X_{n,1},\dots,X_{n,\lfloor n^{1/3+\epsilon}\rfloor-1},X_{n,\lceil n-n^{1/3+\epsilon}\rceil+1},\dots,X_{n,n}$ such that 
		$\P({L}_{\mathbb{S}^1,n,i,m}\neq X_{n,i})\le C n^{-1-3\epsilon/2}$ and $\E[X_{n,i}]=1/n^2$ for all $i$. Then we have
		$$
		\E\left[n\left(\sum_{i=1}^{\lfloor n^{1/3+\epsilon}\rfloor-1}X_{n,i}+\sum_{i=\lceil n-n^{1/3+\epsilon}\rceil+1}^{n}X_{n,i}\right )\right] = 2n^{-2/3+\epsilon},
		$$
		which implies that 
		$$
		n\left(\sum_{i=1}^{\lfloor n^{1/3+\epsilon}\rfloor-1}X_{n,i}+\sum_{i=\lceil n-n^{1/3+\epsilon}\rceil-1}^{n}X_{n,i}\right )\stackrel{\P}{\longrightarrow}0.
		$$
		Also, 
		$$
		\P\left(\sum_{i=1}^{\lfloor n^{1/3+\epsilon}\rfloor-1}{L}_{\mathbb{S}^1,n,i,m}+\sum_{i=\lceil n-n^{1/3+\epsilon}\rceil+1}^{n}{L}_{\mathbb{S}^1,n,i,m}\neq\sum_{i=1}^{\lfloor n^{1/3+\epsilon}\rfloor -1}X_{n,i}+\sum_{i=\lceil n-n^{1/3+\epsilon}\rceil+1}^{n}X_{n,i}\right)\le C n^{-2/3-\epsilon/2}.
		$$
		The proof is completed.
	\end{proof}
	
	
	\bigskip
	\noindent {\bf {\Large Acknowledgments}}
	
	\bigskip
	\noindent The author thanks Professor Jason Schweinsberg for his patient guidance and helpful advice during the planning and development of this article.
	
	
	\bibliographystyle{plain}
	\bibliography{Ref.bib}

\begin{thebibliography}{10}

\bibitem{arratia1979coalescing}
Richard~Alejandro Arratia.
\newblock Coalescing \text{Brownian} motions on the line \text{Ph.D. Thesis,}
  \text{Univ. Wisconsin, Madison.}
\newblock 1979.

\bibitem{birkner2013statistical}
Matthias Birkner, Jochen Blath, and Bjarki Eldon.
\newblock Statistical properties of the site-frequency spectrum associated with
  {$\Lambda$}-coalescents.
\newblock {\em Genetics}, 195(3):1037--1053, 2013.

\bibitem{blath2016site}
Jochen Blath, Mathias~Christensen Cronj{\"a}ger, Bjarki Eldon, and Matthias
  Hammer.
\newblock The site-frequency spectrum associated with {$\Xi$}-coalescents.
\newblock {\em Theoretical Population Biology}, 110:36--50, 2016.

\bibitem{de2007stepping}
Arkendra De and Richard Durrett.
\newblock Stepping-stone spatial structure causes slow decay of linkage
  disequilibrium and shifts the site frequency spectrum.
\newblock {\em Genetics}, 176(2):969--981, 2007.

\bibitem{10.1214/19-AAP1462}
Christina~S. Diehl and G{\"o}tz Kersting.
\newblock {Tree lengths for general $\Lambda $-coalescents and the asymptotic
  site frequency spectrum around the Bolthausen–Sznitman coalescent}.
\newblock {\em The Annals of Applied Probability}, 29(5):2700 -- 2743, 2019.

\bibitem{donnelly2000continuum}
Peter Donnelly, Steven~N Evans, Klaus Fleischmann, Thomas~G Kurtz, and Xiaowen
  Zhou.
\newblock Continuum-sites stepping-stone models, coalescing exchangeable
  partitions and random trees.
\newblock {\em Annals of Probability}, pages 1063--1110, 2000.

\bibitem{durrett2013population}
Rick Durrett.
\newblock Population genetics of neutral mutations in exponentially growing
  cancer cell populations.
\newblock {\em The Annals of Applied Probability}, 23(1):230--250, 2013.

\bibitem{durrett2019probability}
Rick Durrett.
\newblock {\em Probability: Theory and Examples}.
\newblock Cambridge University Press, 2019.

\bibitem{fu1993statistical}
Yun-Xin Fu and Wen-Hsiung Li.
\newblock Statistical tests of neutrality of mutations.
\newblock {\em Genetics}, 133(3):693--709, 1993.

\bibitem{garbit2014exit}
Rodolphe Garbit and Kilian Raschel.
\newblock On the exit time from a cone for \text{Brownian} motion with drift.
\newblock {\em Electronic Journal of Probability}, 19:1--27, 2014.

\bibitem{gunnarsson2021exact}
Einar~Bjarki Gunnarsson, Kevin Leder, and Jasmine Foo.
\newblock Exact site frequency spectra of neutrally evolving tumors: A
  transition between power laws reveals a signature of cell viability.
\newblock {\em Theoretical Population Biology}, 142:67--90, 2021.

\bibitem{johnson2023estimating}
Brian Johnson, Yubo Shuai, Jason Schweinsberg, and Kit Curtius.
\newblock Estimating single cell clonal dynamics in human blood using
  coalescent theory.
\newblock {\em bioRxiv}, 2023.

\bibitem{norris2015weak}
James Norris and Amanda Turner.
\newblock {Weak convergence of the localized disturbance flow to the coalescing
  Brownian flow}.
\newblock {\em The Annals of Probability}, 43(3):935 -- 970, 2015.

\bibitem{schweinsberg2023asymptotics}
Jason Schweinsberg and Yubo Shuai.
\newblock Asymptotics for the site frequency spectrum associated with the
  genealogy of a birth and death process.
\newblock {\em arXiv preprint arXiv:2304.13851}, 2023.

\bibitem{spence2016site}
Jeffrey~P Spence, John~A Kamm, and Yun~S Song.
\newblock The site frequency spectrum for general coalescents.
\newblock {\em Genetics}, 202(4):1549--1561, 2016.

\end{thebibliography}
\end{document}